\documentclass[a4paper,12pt,reqno]{amsart}

\usepackage{amsfonts}
\usepackage{amsmath}
\usepackage{amssymb}
\usepackage{mathrsfs}
\usepackage{nomencl}
\usepackage{delarray}
\usepackage{blkarray}

\usepackage{hhline}
\usepackage[colorlinks]{hyperref}

\numberwithin{equation}{section}
\theoremstyle{plain}

\newtheorem{thm}{Theorem}[section]
\newtheorem{prop}[thm]{Proposition}
\newtheorem{cor}[thm]{Corollary}
\newtheorem{lem}[thm]{Lemma}

\theoremstyle{definition}
\newtheorem{defn}[thm]{Definition}
\newtheorem{assumption}[thm]{Assumption}
\newtheorem{rem}[thm]{Remark}

\newcommand{\inlinemaketitle}{{\let\newpage\relax\maketitle}}
\newcommand{\be}{\begin{equation}}
\newcommand{\ee}{\end{equation}}

\def\eps{\varepsilon}
\def\HS{{\mathtt{HS}}}

\def\TT{{\mathbb T}}
\def\NN{{\mathbb N}}
\def\dualSU2{\frac12\NN_0}
\def\RR{{\mathbb R}}
\def\ZZ{{\mathbb Z}}
\def\C{{\mathbb C}}
\def\Gh{{\widehat{G}}}
\def\dpi{{d_\pi}}
\def\kpi{{k_\pi}}
\def\FT{{\mathscr F}}

\def\B{{\mathcal B}}
\def\M{\mathcal{M}}
\def\Dcal{{\mathcal D}}

\def\p#1{{\left({#1}\right)}}

\def\jp#1{{\left\langle{#1}\right\rangle}}

\def\wt#1{{\widetilde{#1}}}

\def\SU2{{\rm SU(2)}}

\DeclareMathOperator{\Tr}{Tr}

\begin{document}
\title[Net spaces on lattices and Hardy-Littlewood type inequalities]
{Net spaces on lattices, Hardy-Littlewood type inequalities, and their converses}

\author[Rauan Akylzhanov]{Rauan Akylzhanov}
\address{Rauan Akylzhanov:
  \endgraf
  Department of Mathematics
  \endgraf
  Imperial College London
  \endgraf
  180 Queen's Gate, London SW7 2AZ
  \endgraf
  United Kingdom
  \endgraf
  {\it E-mail address} {\rm r.akylzhanov14@imperial.ac.uk}
}

\author[Michael Ruzhansky]{Michael Ruzhansky}
\address{
  Michael Ruzhansky:
  \endgraf
  Department of Mathematics
  \endgraf
  Imperial College London
  \endgraf
  180 Queen's Gate, London SW7 2AZ
  \endgraf
  United Kingdom
  \endgraf
  {\it E-mail address} {\rm m.ruzhansky@imperial.ac.uk}
  }

\thanks{The second
 author was supported in parts by the EPSRC Grant EP/K039407/1 and by the Leverhulme Grant RPG-2014-02.
 No new data was collected or generated during the course of the research.}
\date{\today}


\begin{abstract}
       We introduce abstract net spaces on directed sets
	and prove their embedding and interpolation properties.
	Typical examples of interest are lattices of irreducible unitary representations of compact Lie groups
	and of class I representations with respect to a subgroup.
	As an application, we prove Hardy-Littlewood type inequalities and their converses
	on compact Lie groups and on compact homogeneous manifolds.
\end{abstract}
\maketitle

\section{Introduction}

In \cite{HL}, Hardy and Littlewood proved the following estimate on the circle $\TT$, relating 
$L^p$-norms of a function and its Fourier coefficients:
\begin{equation}
\label{H_L_inequality-0}
	\sum_{m \in \ZZ}{(1+|m|)}^{p-2}|\widehat{f}(m)|^{p} \leq C\|f\|^p_{L^p(\TT)},\quad 1<p\leq 2.
\end{equation}
They also argued this to be a suitable extension of the Plancherel's identity to the setting of
$L^p$-spaces. In fact, they also proved that the inequality becomes an equivalence provided
that the Fourier coefficients $\widehat{f}(m)$ are monotone. 

By duality, we readily obtain the corresponding inequality also in the range $2\leq p<\infty$,
namely, we also have the estimate
\begin{equation}
\label{H_L_inequality-01}
	\|f\|^p_{L^p(\TT)}\leq C_{p}^{\prime}\sum_{m\in\ZZ} (1+|m|)^{p-2}|\widehat{f}(m)|^p,
	\quad 2\leq p<\infty.
\end{equation}

In this paper we are interested in 
inequalities of Hardy-Littlewood type.
For example, let $1<p<\infty$ and $f\in L^p(\mathbb{T})$, and suppose that 
$$
f \sim \sum_{m\in\mathbb{Z}}\widehat{f}(m)e^{2\pi imx}.
$$
Then it was shown in \cite{NED} that we have
\begin{equation}
\label{necess_T-00}
	\sum^{\infty}_{k=1}k^{p-2}\left(\sup_{\substack{e\in M_0 \\ |e|\geq k}}
	\frac1{|e|}\left|\sum_{m\in e}\widehat{f}(m)\right|\right)^p\leq C \|f\|^p_{L^p(\TT)}, \quad 1<p<\infty,
\end{equation}
where $M_0$ is the set of all finite arithmetic progressions in $\mathbb Z$.
Especially in the range $2\leq p<\infty$ this gives a converse estimate to the Hardy-Littlewood estimate
\eqref{H_L_inequality-01}.

Net function spaces $N_{p,q}$ on $\ZZ^n$ and $\RR^n$  were introduced  in \cite{Nur1998} as a 
machinery to prove the inequalities of type \eqref{necess_T-00} 
for Fourier coefficients of functions on $\mathbb{T}^n$ and $\mathbb{R}^n$. 
Since then, they found other applications as well:
we can refer to \cite{NED,Nur1998,Nursultanov-Tikhonov:JGA-2011} for some applications of
these spaces to questions of harmonic analysis and approximation theory.

Since the unitary dual $\widehat{\TT^n}$ of a compact abelian Lie group $\TT^n$ is isomorphic
to $\ZZ^n$, i.e.
$$
	\widehat{\TT^n}
	=\{e^{2\pi i k\cdot x}\}_{k\in\ZZ^n}\ni e^{2\pi i k\cdot x}
	\longleftrightarrow k\in\ZZ^n,
$$
we can consider $N_{p,q}(\ZZ^n)$ as a net space $N_{p,q}(\widehat{\TT^n})$ on the lattice $\ZZ^n$.
It turns out that the theory of net spaces $N_{p,q}$ can be extended to arbitrary lattices provided 
we make certain rather natural assumptions. 
In this paper we develop this abstract setting to be able to use the notion of a net space on the 
unitary dual of a compact
Lie group and on the lattice of its class I representations. In addition to a suitable definition, 
for our purposes we need to prove their embedding and
interpolation properties. As it is common, 
such technique allows one to derive `strong' estimates from `weak' ones by interpolation.

As an application of these results, we obtain Hardy-Littlewood type inequalities on
compact Lie groups and compact homogeneous manifolds, also providing the inverses to
the Hardy-Littlewood inequalities that were recently obtained in \cite{ANR2015}.
In Corollary \ref{COR:HL-converse-SU2} we calculate an explicit example of such
an inverse (to the Hardy-Littlewwod inequality) in the case of the group SU(2).

The obtained results also yield a noncommutative version to known estimates of the type
\eqref{necess_T-00} on a circle $\TT$.

In Section \ref{SEC:net_space_on_lattices} we develop the notion of net spaces on rather general
lattices and prove their main properties (interpolation and embedding).
In Section \ref{SEC:HL-converse} we apply these results to obtain inverses to
known Hardy-Littlewood inequalities in the settings of compact Lie groups and compact
homogeneous spaces.

\section{Net spaces on lattices}
\label{SEC:net_space_on_lattices}

 Let $\Gamma$ be a discrete set. We assume that there exists such partial order $\prec$ on $\Gamma$ that every two elements in $\Gamma$ are comparable under $\prec$ in $\Gamma$. In addition, we suppose that $\Gamma$ is bounded from below. In other words, there exists an element $1\in \Gamma$ such that $1\prec \pi$ for all 
$\pi\in\Gamma$. The partial order $\prec$ on $\Gamma$ makes it possible to define the notion of a net which was first introduced by Smoth and Moore in \cite{MS1922}. This `net' is different from `net' in net spaces $N_{p,q}$.
Let $T$ be a topological space.  A net $a$ in $T$ is a function from $\Gamma$ to $T$, i.e.
$$
a=\{a_{\pi}\}_{\pi\in\Gamma}\colon \Gamma\ni\pi \mapsto a_{\pi}\in T.
$$

We consider two nets $\delta=\{\delta_{\pi}\}_{\pi\in\Gamma}$ and $\kappa=\{\kappa_{\pi}\}_{\pi\in\Gamma}$ in $T=\NN$, i.e.
\begin{eqnarray*}
\Gamma\ni \pi \mapsto \delta_{\pi}\in\NN,
\\
\Gamma\ni \pi \mapsto \kappa_{\pi}\in\NN.
\end{eqnarray*}

We turn $\Gamma$ into a $\sigma$-finite measure space by introducing a measure 
\begin{equation}\label{EQ:nu}
\nu_{\Gamma}(Q):=\sum\limits_{\substack{\theta\in Q }}\delta_{\theta}\kappa_{\theta},
\end{equation}
where $Q$ is an arbitrary  subset of $\Gamma$. We denote this measure space by $(\Gamma,\nu_{\Gamma})$.
We denote by $\Sigma$ the space of matrix-valued sequences on $\Gamma$ that will be realised via
$$\Sigma:=\left\{ h=\{h(\pi)\}_{\pi\in\Gamma}, 
h(\pi)\in \mathbb{C}^{\kappa_{\pi}\times \delta_{\pi}}\right\}.$$
The $\ell^p$ spaces on $\Sigma$ can be defined, for example, motivated by the Fourier analysis on compact homogeneous spaces (see \cite{RT}), in the form 
$$
\|h\|_{\ell^p(\Gamma,\nu_{\Gamma},\Sigma)}
:=
\left(
\sum\limits_{\pi\in\Gamma}
\delta_{\pi}\kappa_{\pi}^{p(\frac1p-\frac12)}
\|h(\pi)\|^p_{\HS}
\right)^{\frac1p},\quad h\in\Sigma.
$$
Sometimes, we can abbreviate this by writing $\ell^p(\Gamma,\nu_{\Gamma}),\ell^p(\Gamma)$ or $\ell^p$.

If we put $\Gamma=\Gh$, where $\Gh$ is the unitary dual of a compact Lie group $G$, then Fourier transform can be regarded as an operator mapping a function $f\in L^p(G)$ to the matrix-valued sequence $\widehat{f}=\{\widehat{f}(\pi)\}_{\pi\in\Gh}$  of the Fourier coefficients $\widehat{f}(\pi)\in\C^{\dpi\times\dpi}$ given by $\widehat{f}(\pi)=\int\limits_Gf(u)\pi(u)^*\,du$.
Let us denote by $\Gh_0$ the subset of $\Gh$ of representations that are class I with respect to some  subgroup of $G$.
For $\Gamma=\Gh_0$ we put $\delta_{\pi}=\dpi$ and $\kappa_{\pi}=\kpi$,  these spaces thus coincide with the $\ell^p(\Gh_0)$ spaces introduced in \cite{RT}.  
See \cite{NRT2014} for the definition of $\dpi$ and $\kpi$ but
these notations will also be explained in detail in Section \ref{SEC:HL-converse}. 
See also \cite{NRT:FAA} for the group setting.

\medskip
It can be easily verified that the following formula holds true.
\begin{rem}
\label{REM:ell_p_Gamma_space_duality}
Let $1<p<\infty$. For $\ell^p(\Gamma,\nu_{\Gamma},\Sigma)$, we have
\begin{equation}
\label{EQ:ell_p_Gamma_space_duality}
	\|h\|_{\ell^p(\Gamma,\nu_{\Gamma})}
	=
	\sup_{\substack{g\in\ell^{p'}(\Gamma,\nu_{\Gamma})\\ g \neq 0}}
	\frac{\left|\sum\limits_{\pi\in\Gamma}\delta_{\pi}\Tr[h(\pi)g(\pi)^*]\right|}{\|g\|_{\ell^{p'}(\Gamma,\nu_{\Gamma})}},
\end{equation}
where $\frac1{p'}+\frac1p=1,$ and $h,g\in\Sigma$. The matrix $g(\pi)^*$ is the Hermitian conjugate of $g(\pi)$ and $\Tr$ is the matrix trace.
\end{rem}


We now give a definition of net spaces.

\begin{defn} 
\label{DEF:Npq-lattice}
Let $\lambda=\{\lambda_{\pi}\}_{\pi\in\Gamma}$ be an arbitrary positive sequence over $\Gamma$.
Denote by $\M$ a fixed arbitrary collection of finite subsets of  $\Gamma$. 
	Given a family of complex matrices $F=\{F(\pi)\}_{\pi\in\Gamma}$,
	$F(\pi)\in\C^{\kappa_{\pi}\times \delta_{\pi}}$ and $1\leq p<\infty,\,1\leq q\leq \infty$, define
	\[
	\|F\|_{N_{p,q}(\Gamma,\M)}
	:=
	\begin{cases}
	\displaystyle
	\left(
	\sum_{\pi\in\Gamma}
	\left(
	\lambda_{\pi}^{\frac{1}{p}}
	\overline{F}[\lambda_{\pi},\M]
	\right)^q\frac{\delta_{\pi}\kappa_{\pi}}{\lambda_{\pi}}
	\right)^{\frac1q},& \text{ \quad if\quad} q<\infty,\\
	\displaystyle
	\sup_{\substack{\pi\in\Gamma}}\lambda^{\frac{1}{p}}_{\pi}
	\overline{F}[\lambda_{\pi},\M],& \text{\quad if \quad} q = \infty,
	\end{cases}
	\]
	where 
	\begin{equation}
\label{DEF:averaging}
	\overline{F}[\lambda_{\pi},\M]
	:=
	\sup\limits_{\substack{Q\in \M \\\nu_{\Gamma}(Q)\geq \lambda_{\pi}}}
	\frac
		{
			1
		}
		{
			\nu_{\Gamma}(Q)
		}
	\left|
	\sum\limits_{\theta\in Q}d_{\theta}\Tr F(\theta)
	\right|
	,
	\end{equation}
and $\Tr F(\theta)=\sum\limits^{\min(\kappa_{\theta},\delta_{\theta})}_{j=1}F(\theta)_{jj}$.
\end{defn}
We call $\overline{F}[\lambda,\M]=\{\overline{F}[\lambda_{\pi},\M]\}_{\pi\in\Gamma}$ defined by \eqref{DEF:averaging} 
{\it the averaging of $F=\{F(\pi)\}_{\pi\in\Gamma}$ with respect to $\M$}. 
Sometimes we may drop writing $\M$ to simplify the notation.
In comparison to the well-known maximal function, the averaging function allows one to capture the oscillation 
properties of sequences/functions/nets.

In general, different partial orders $\prec_1$ and $\prec_2$ on $\Gamma$ will give different $N_{p,q}$ spaces  on $\Gamma$.
\section{Hardy-Littlewood type inequalities}
\label{SEC:HL-converse}

In this section we apply net spaces $N_{p, q}(\Gamma)$ on an  ordered lattice $\Gamma$ to establish new inequalities relating functions on $G$ and their Fourier coefficients.
In \cite{NED} the following theorem was established. In the sequel $L^{p,q}$ denotes the Lorenz space.
\begin{thm}[\cite{NED}]\label{NED}
\label{THM:NED}
Let $1\leq p<\infty,1\leq q \leq \infty$ and $f\in L^{p,q}(\TT)$. Suppose that 
$$
f \sim \sum_{m\in\mathbb{Z}}\widehat{f}(m)e^{2\pi im x}.
$$
Then we have
\begin{equation}
\label{necess_T-0}
	\sum\limits_{k\in \NN}
	\left(
		k^{\frac1{p'}}
		\sup_{\substack{e\in \M_0 \\ |e|\geq k}}
		\frac1{|e|}\left|\sum_{m\in e}\widehat{f}(m)\right|
	\right)^q
	\frac1{k}
	\leq C \|f\|^p_{L^{p,q}(\TT)},
\end{equation}
or equivalently in terms of net spaces $N_{p,q}(\ZZ,\M_0)$
\begin{equation}
\label{EQ:Npq-estimate}
\|\widehat{f}\|_{N_{p',q}(\ZZ,\M_0)}
\lesssim
\|f\|_{L^{p,q}(\TT)},
\end{equation}
where $\M_0$ is the set of all finite arithmetic progressions in $\mathbb Z$, with the constant in \eqref{EQ:Npq-estimate} independent of  $f$.
\end{thm}
\begin{rem}Since $N_{p,q}(\ZZ,\M_0)$ are interpolation spaces (\cite[Theorem 1, p.88]{NED}),  in order to establish \eqref{EQ:Npq-estimate} it is sufficient to establish a `weak' inequality
\begin{equation}
\|\widehat{f}\|_{N_{p',\infty}(\ZZ,\M_0)}
\lesssim
\|f\|_{L^p(\TT)}.
\end{equation}
\end{rem}
In \cite[Proposition 1]{NED} it has been established that if the class $\M$ contains all finite subsets then $N_{p,q}$ coincides, up to constant, with the Lorenz space $L^{p,q}:$
\begin{thm}[\cite{NED}]
\label{THM:Npq-1p-2}
 Let $1<p<+\infty,1\leq q\leq \infty$, and let $\M_1$ be the set of all finite subsets of $\ZZ$. Then we have
\begin{equation}
N_{p,q}(\ZZ,\M_1)\cong L^{p,q}(\ZZ).
\end{equation}
\end{thm}
For $1<p<2$, inequality \eqref{EQ:Npq-estimate} can be refined:
\begin{rem} Let $1<p<2$. We have by Theorem \ref{THM:Net_Space_Embedd} below and Theorem \ref{THM:Npq-1p-2}, that
\begin{equation}
\label{EQ:comparison}
\|\widehat{f}\|_{N_{p',p}(\ZZ,\M_0)}
\lesssim
\|\widehat{f}\|_{N_{p',p}(\ZZ,\M_1)}
\cong
\|\widehat{f}\|_{L^{p',p}(\ZZ)}
\lesssim
\|f\|_{L^p(\TT)}.
\end{equation}
\end{rem}
 The last inequality in \eqref{EQ:comparison} is essentially the Hardy-Littlewood inequality for Fourier coefficients, see \cite{NED} for the details.
The application of $N_{p,q}(\Gamma)$ with $\Gamma=\Gh_0$ yields the extension of Theorem \ref{THM:Npq-1p-2} to the setting of compact homogeneous manifolds $G/K$. In addition, in Theorem \ref{THM:characterization} we characterise those classes $\M$ for which inequality 
\begin{equation}
\label{EQ:N_pq-L_p0}
\|\widehat{f}\|_{N_{p',\infty}(\Gh_0,\M)}
\leq
C
\|f\|_{L^p(G/K)},\quad f\in L^p(G/K)
\end{equation}
holds.
The characterisation is given in terms of of the behaviour of Dirichlet kernel's norms $\|D_Q\|_{L^p(G/K)}$, $Q\in\M$.

To motivate the formulation, we start with a compact Lie group $G$. 
Identifying a representation $\pi$ with its equivalence class and choosing some
bases in the representation spaces of degree $\dpi$, we can think of $\pi\in\Gh$ as a 
mapping $\pi:G\to\C^{\dpi\times\dpi}$. For $f\in L^{1}(G)$, we define its
Fourier transform at $\pi\in\Gh$ by
$$
(\FT_{G} f)(\pi)\equiv \widehat{f}(\pi):=\int_{G} f(u)\pi(u)^{*} du,
$$
where $du$ is the normalised Haar measure on $G$.
This definition can be extended to distributions $f\in {\mathcal D}'(G)$, and
the Fourier series takes the form
\begin{equation}\label{EQ:FS}
f(u)=\sum_{\pi\in\Gh} \dpi \Tr\p{\pi(u) \widehat{f}(\pi)}.
\end{equation}
The Plancherel identity on $G$ is given by
\begin{equation}
\label{plancherel}
\|f\|^{2}_{L^{2}(G)}=\sum_{\pi\in\Gh} \dpi \|\widehat{f}(\pi)\|_{\HS}^{2}=: 
\|\widehat{f}\|_{\ell^{2}(\Gh)}^{2},
\end{equation}
yielding the Hilbert space $\ell^{2}(\Gh)$.
The Fourier coefficients of functions and distributions on $G$ take values in the space
\begin{equation}\label{EQ:Sigma}
\Sigma=\left\{ \sigma=(\sigma(\pi))_{\pi\in\Gh}: \sigma(\pi)\in \C^{\dpi\times\dpi} \right\}.
\end{equation}
The $\ell^p$-spaces on the unitary dual $\Gh$ have been 
developed in \cite{RT} based on fixing the Hilbert-Schmidt norms.
Namely, for $1\leq p<\infty$, we define the space
$\ell^p(\Gh)$ by the norm
\begin{equation}\label{EQ:Lp}
\|\sigma\|_{\ell^p(\Gh)}:=\p{ \sum_{\pi\in\Gh} d_\pi^{p\p{\frac2p-\frac12}} \|\sigma(\pi)\|_{\HS}^p}^{1/p},
\; \sigma\in\Sigma, \; 1\leq p<\infty,
\end{equation}
where $\|\cdot\|_{\HS}$ denotes the Hilbert-Schmidt matrix norm i.e.
$$
\|\sigma(\pi)\|_{\HS}:=\left(\Tr(\sigma(\pi)\sigma(\pi)^*)\right)^{\frac12}.
$$
It was shown in \cite[Section 10.3]{RT} that, among other things, these are interpolation spaces,
and that the Fourier transform $\FT_{G}$ and its inverse $\FT_{G}^{-1}$ satisfy the
Hausdorff-Young inequalities in these spaces. We can also refer to
\cite{Ruzhansky+Turunen-IMRN} for pseudo-differential extensions of the Fourier analysis
on both compact Lie groups and homogeneous manifolds.

%

We now describe the setting of Fourier coefficients on a compact homogeneous
manifold $M$ following 
\cite{Dasgupta-Ruzhansky:Gevrey-BSM} or \cite{NRT2014},
and referring for further details with
proofs to Vilenkin \cite{Vilenkin:BK-eng-1968} or to
Vilenkin and Klimyk \cite{VK1991}.

Let $G$ be a compact motion group of $M$ and let $K$ be the stationary
subgroup of some point.
Alternatively, we can start with a compact Lie group $G$ with
a closed subgroup $K$, and identify $M=G/K$ as an analytic manifold in a canonical way.
We normalise measures so
that the measure on $K$ is a probability one.
Typical examples are the spheres ${\mathbb S}^{n}={\rm SO}(n+1)/{\rm SO}(n)$
or complex spheres $\mathbb C{\mathbb S}^{n}={\rm SU}(n+1)/{\rm SU}(n)$.

Let us denote by $\Gh_{0}$ the subset of $\Gh$ of representations that are
class I with respect to the subgroup $K$. This means that 
$\pi\in\Gh_{0}$ if $\pi$ has at least one non-zero invariant vector $a$ with respect
to $K$, i.e. that
$$\pi(h)a=a \; \textrm{ for all } \;h\in K.$$ 
Let $\B_{\pi}$ denote the space of these invariant vectors and let  
$$\kpi:=\dim\B_{\pi}.$$
Let us fix an orthonormal basis in the representation space of $\pi$ so that
its first $\kpi$ vectors are the basis of $B_{\pi}.$
The matrix elements $\pi(x)_{ij}$, $1\leq j\leq \kpi$,
are invariant under the right shifts by $K$.

We note that if $K=\{e\}$ so that $M=G/K=G$ is the Lie group, we have
$\Gh=\Gh_{0}$ and $\kpi=\dpi$ for all $\pi$. As the other extreme, if
$K$ is a massive subgroup of $G$, i.e., if for every $\pi$ there is precisely one invariant vector
with respect to $K$, we have $k_{\pi}=1$ for all $\pi\in\Gh_{0}.$
This is, for example, the case for the spheres $M={\mathbb S}^{n}.$
Other examples can be found in Vilenkin \cite{Vilenkin:BK-eng-1968}.

We can now identify functions on $M=G/K$ with functions on 
$G$ which are constant on left cosets with respect to $K$. 
Then, for a function $f\in C^{\infty}(M)$ we can
recover it by the Fourier series of its canonical
lifting $\wt{f}(g):=f(gK)$ to $G$,
$\wt{f}\in C^{\infty}(G)$, and the Fourier
coefficients satisfy  $\widehat{\wt{f}}(\pi)=0$ for all representations
with $\pi\not\in\Gh_{0}$.
Also, for class I representations $\pi\in\Gh_{0}$ we have 
$\widehat{\wt{f}}(\pi)_{{ij}}=0$ for $i>\kpi$.

With this, we can write the Fourier series of $f$ (or of $\wt{f}$, but we identify these) 
in terms of
the spherical functions $\pi_{ij}$ of the representations
$\pi\in\Gh_{0}$, with respect to the subgroup
$K$. 
Namely, the Fourier series \eqref{EQ:FS} becomes
\begin{equation}\label{EQ:FSh}
f(x)=\sum_{\pi\in\Gh_{0}} \dpi \sum_{i=1}^{\dpi}\sum_{j=1}^{\kpi}
\widehat{f}(\pi)_{ji}\pi(x)_{ij}=
\sum_{\pi\in\Gh_{0}} \dpi \Tr (\widehat{f}(\pi) \pi(x)),
\end{equation}
where, in order to have the last equality, we adopt the convention 
of setting $\pi(x)_{ij}:=0$ for all $j>\kpi$, for
all $\pi\in\Gh_{0}$.
With this convention the matrix
$\pi(x)\pi(x)^{*}$ is diagonal with the first $\kpi$ diagonal entries equal to one and
others equal to zero, so that we have
\begin{equation}\label{EQ:xi-HS}
\|\pi(x)\|_{\HS}=\sqrt{k_\pi} \textrm{ for all } \pi\in\Gh_{0}, \; x\in G/K.
\end{equation}
Following \cite{Dasgupta-Ruzhansky:Gevrey-BSM},
we will say that {\it the collection of Fourier coefficients
$\{\widehat{f}(\pi)_{ij}: \pi\in\Gh, 1\leq i,j\leq \dpi\}$ is 
of class I with respect to $K$} if
$\widehat{f}(\pi)_{ij}=0$ whenever $\pi\not\in\Gh_{0}$ or
$i>\kpi.$ By the above discussion, if the collection of Fourier
coefficients is of class I with respect to $K$, then the expressions
\eqref{EQ:FS} and \eqref{EQ:FSh} coincide and yield a function
$f$ such that $f(xh)=f(h)$ for all $h\in K$, so that this function becomes
a function on the homogeneous space $G/K$. 

For the space of Fourier coefficients of class I we define the analogue of the
set $\Sigma$ in \eqref{EQ:Sigma} by
\begin{equation}\label{EQ:SigmaGK}
\Sigma(G/K):=\{\sigma:\pi\mapsto\sigma(\pi): \;
\pi\in\Gh_{0},\; \sigma(\pi)\in\C^{\dpi\times\dpi}, \; \sigma(\pi)_{ij}=0
\textrm{ for } i>\kpi\}.
\end{equation}
In analogy to \eqref{EQ:Lp},
we can define the Lebesgue spaces $\ell^{p}(\Gh_{0})$
by the following norms which we will apply to Fourier coefficients
$\widehat{f}\in\Sigma(G/K)$ of $f\in\Dcal'(G/K)$.
Thus, for $\sigma\in\Sigma(G/K)$ we set
\begin{equation}\label{EQ:Lp-sigmaGK}
\|\sigma\|_{\ell^{p}(\Gh_{0})}:=\left(\sum_{\pi\in\Gh_{0}} \dpi
k_{\pi}^{p(\frac{1}{p}-\frac12)}
\|\sigma(\pi)\|_{\HS}^{p}\right)^{1/p},\; 1\leq p<\infty.
\end{equation}
In the case $K=\{e\}$, so that $G/K=G$, these spaces coincide with
those defined by \eqref{EQ:Lp} since $\kpi=\dpi$ in this case.
Again, by the same argument as that in \cite{RT},
these spaces are interpolation spaces and the Hausdorff-Young
inequality holds for them. We refer to
\cite{NRT2014} for some more details on
these spaces.

%

Let $\M$ be an arbitrary collection of finite subsets of $\Gh_0$.  Denote by $\M_1$ the collection of all finite subsets of $\Gh_0$.
For $Q\subset \M_1$, the measure $\nu_\Gamma$ with $\Gamma=\Gh_0$ in \eqref{EQ:nu} is finite and is given by
\begin{equation}\label{EQ:nu1}
\nu_{\Gamma}(Q)=\sum\limits_{\substack{\theta\in Q }} d_{\theta} k_{\theta},\quad Q\in \M_1\subset\Gh_0.
\end{equation}
Let $Q\in\M$ and write $$D_Q(x):=\sum\limits_{\pi\in Q}\dpi\Tr[\pi(x)].$$
Denote by $D(\M)$ the set of all Dirichlet kernels with their spectrum embedded in some $Q\in\M$, i.e.
$$
D(\M)
:=
\{D_{Q}(x),Q\in\M\}.
$$

Now, we can characterise those $\M$ for which
$$
\|\widehat{f}\|_{N_{p',\infty}(\Gh_0,\M)}
\leq
C
\|f\|_{L^p(G/K)},\quad f\in L^p(G/K),
$$
via a certain condition on the size of the $\|D_Q\|_{L^{p'}(G/K)}$ norm.
In the theorem below we have $\Gamma=\Gh_0$ and the measure
$\nu_\Gamma=\nu_{\Gh_0}$ is given by \eqref{EQ:nu1}. In the sequel we can use both notations
in the case of homogeneous manifolds $G/K$.
\begin{thm} 
\label{THM:characterization}
Let $1<p\leq \infty$ and let $\M$ be an arbitrary collection of finite subsets of $\Gh_0$.
Then
\begin{equation}
\label{EQ:N_pq-L_p}
\|\widehat{f}\|_{N_{p',\infty}(\Gh_0,\M)}
\leq
C
\|f\|_{L^p(G/K)},\text{ for all } f\in L^p(G/K),
\end{equation}
if and only if
\begin{equation}
\label{EQ:DQ-condition}
C_{p\M}
:=
\sup_{\pi\in\Gh_0}
\lambda^{\frac{1}{p'}}_{\pi}
\sup_{\substack{
Q\in \M
\\ 
\nu_{\Gamma}(Q)
\geq 
\lambda_{\pi}
}}
\frac1{\nu_{\Gamma}(Q)}
\|
D_Q
\|_{L^{p'}(G/K)}
<+\infty,
\end{equation}
with $\nu_\Gamma=\nu_{\Gh_0}$, and  $\lambda_{\pi}$ is the sequence used in the Defintion \ref{DEF:Npq-lattice}.
\end{thm}
In Proposition \ref{PROP:char-implies-TT^n}
we will check in the sequel that condition \eqref{EQ:DQ-condition} for all indices $1<p<\infty$ 
is satisfied in the example of the tori $\TT^n$
if we take $\lambda_\pi$ to be the sequence of the eigenvalues of the Laplacian counted with multiplicities.
In Theorem \ref{THM:HL-Npq-version}, verifying condition \eqref{EQ:DQ-condition},
 we will give an unconditional version of Theorem \ref{THM:characterization} for
the range of indices $1<p\leq 2$ on general compact homogeneous manifolds based on the interpolation properties of
net spaces to be established in the next section.
In Corollary \ref{COR:HL-converse-SU2} we give an example on the
group SU(2), again if we take $\lambda_\pi$ to be the sequence of the eigenvalues of the Laplacian counted with multiplicities,
yielding an inverse to the Hardy-Littlewood inequality there.

\begin{rem} It follows from the proof that  $C_{p\M}\leq C$ and inequality \eqref{EQ:N_pq-L_p} holds true for $C=C_{p\M}$.
\end{rem} 
The interpolation properties of $N_{p,q}$ spaces allow us to formulate and prove a version of the Hardy-Littlewood inequality in terms of $N_{p,q}$ spaces.
\begin{thm} 
\label{THM:HL-Npq-version}
Let $1<p\leq 2$, $1\leq q\leq \infty$, and let $\M_1$ be the collection of all finite subsets of $\Gh_0$. 
 Then we have
\begin{equation}
\|\widehat{f}\|_{N_{p',q}(\Gh_0,\M_1)}
\leq
C_{p,q}\|f\|_{L^{p,q}(G/K)}.
\end{equation}
\end{thm}
We give a corollary of Theorem \ref{THM:HL-Npq-version} on $\SU2$. In this case, we simplify general notation. It can be shown that the unitary dual $\widehat{\SU2}$ can be `labelled' by the set of half-integers $\frac12\NN_0$. Thus, we write $\widehat{f}(l)$ for the Fourier coefficient with respect to the element $t^l\in\widehat{\SU2}$ associated with $l\in\frac12\NN_0$.
Here $d_{t^l}=2l+1$ so that $\widehat{f}(l)\in\mathbb C^{(2l+1)\times (2l+1)}$, $l\in \frac12\NN_0$.

\begin{cor} 
\label{COR:HL-converse-SU2}
Let $1<p\leq 2$. Then we have
\begin{multline}
\label{EQ:HL-converse-SU2}
\sum\limits_{\xi\in\frac12\NN_0}(2\xi+1)^{\frac{5p}2-4}
\left(
\sup_{\substack{k\in\frac12\NN_0 \\ (2k+1) \geq (2\xi+1)}}
\frac1{(2k+1)^3}
\left|
\sum\limits_{\substack{l\in\frac12\NN_0 \\ 2l+1\leq 2k+1}}
(2l+1)\Tr\widehat{f}(l)
\right|
\right)^p
\\\leq
C_p
\|f\|^p_{L^p(\SU2)}.
\end{multline}
\end{cor}
We will prove Corollary \ref{COR:HL-converse-SU2} in Section \ref{SEC:netspaces}
after Remark \ref{REM:Assumption_3_G}.

\begin{proof}[Proof of Theorem \ref{THM:HL-Npq-version}]
For $1<p\leq 2$ the condition \eqref{EQ:DQ-condition} is satisfied with $\M=\M_1$.
Indeed, by Hausdorff-Young inequality, with $\Gamma=\Gh_0$ and 
$\nu_{\Gamma}(Q)=\sum_{\pi\in Q} d_\pi k_\pi$, $Q\subset\Gh_0$, we have 
\begin{equation}
\|D_Q\|_{L^{p'}(G/K)}
\leq
\nu_{\Gamma}(Q)^{\frac1p},\quad  Q\in\M_1,\; 1<p\leq 2.
\end{equation}
Then, we get
\begin{equation}
C_{p\M_1}=
\sup_{\pi\in\Gh_0}
\lambda^{\frac{1}{p'}}_{\pi}
\sup_{\substack{
Q\in \M_1
\\ 
\nu_{\Gamma}(Q)
\geq 
\lambda_{\pi}
}}
\frac1{\nu_{\Gamma}(Q)}
\|
D_Q
\|_{L^{p'}(G/K)}
\leq
\sup_{\pi\in\Gh_0}
\lambda^{\frac{1}{p'}}_{\pi}
\sup_{\substack{
Q\in \M_1
\\ 
\nu_{\Gamma}(Q)
\geq 
\lambda_{\pi}
}}
\frac1{\nu_{\Gamma}(Q)^{\frac1{p'}}}
=
1.
\end{equation}
This proves that the condition \eqref{EQ:DQ-condition} is satisfied. 
Thus, the application of Theorem \ref{THM:characterization} yields
\begin{equation*}
\|\widehat{f}\|_{N_{p',\infty}(\Gh_0,\M_1)}
\leq
C_{p\M_1}
\|f\|_{L^p(G/K)},\quad 1<p\leq 2.
\end{equation*}
Let $1<p_1<p<p_2\leq 2$. Then interpolating between two inequalities
\begin{eqnarray*}
\|\widehat{f}\|_{N_{{p'_1},\infty}(\Gh_0,\M_1)}
\leq
C_{{p_1}\M_1}
\|f\|_{L^{p_1}(G/K)},
\\
\|\widehat{f}\|_{N_{{p'_2},\infty}(\Gh_0,\M_1)}
\leq
C_{p\M_1}
\|f\|_{L^{p_2}(G/K)}.
\end{eqnarray*}
 (see Theorem \ref{net_space_interpolation} below), we obtain
\begin{multline*}
\|\widehat{f}\|_{N_{p',q}(\Gh_0,\M_1)}
\leq
\|\widehat{f}\colon (N_{p_1,\infty}(\Gh_0,\M_1),N_{p_1,\infty}(\Gh_0,\M_1))_{\theta,q}\|
\\\leq
C_{p,q}
\|f\colon (L^{p_1}(G/K),L^{p_2}(G/K))_{\theta,q}\|
=
C_{p,q}
\|f\|_{L^{p,q}(G/K)},
\end{multline*}
where in the last equality we used the fact that $L^{p,q}(G/K)$ are interpolation spaces.
This completes the proof.
\end{proof}
Now, we show that Theorem \ref{THM:characterization} includes as a particular case Theorem \ref{THM:NED}.
\begin{prop} 
\label{PROP:char-implies-TT^n}
Let $1<p<\infty$, let $M_{a}$ be the set of all finite arithmetic progressions in $\ZZ^n$, $G=\TT^n,\lambda_{e^{2\pi i m\cdot x}}=m$,
$d_{e^{2\pi i m\cdot x}}=k_{e^{2\pi i m\cdot x}}=1$ for $m\in\mathbb Z^n$,
and hence $\nu_{\ZZ^n}(Q)=|Q|$, $Q\in \M_a$. Then
Theorem \ref{THM:characterization} implies Theorem \ref{THM:NED}.
\end{prop}
Here $|\cdot|$ is the counting measure.
\begin{proof}[Proof of Proposition \ref{PROP:char-implies-TT^n}]
We show that condition \eqref{EQ:DQ-condition} holds true for this case.
Indeed, using the $L^p$-space duality, we have
\begin{equation}
\label{EQ:Lp-space_duality}
\|D_Q\|_{L^{p'}(\TT^n)}
=
\sup_{\substack{f\in L^p(\TT^n) \\ f \neq 0}}
\frac{\left|(D_Q,f)_{L^2(\TT^n)}\right|}{\|f\|_{L^p(\TT^n)}}.
\end{equation}

By the Hardy-Littlewood rearrangement inequality (see \cite[p.44 Theorem 2.2]{BeSh1988}), we obtain
\begin{equation}
\left|
\left(D_Q,f\right)_{L^2(\TT^n)}
\right|
\leq
\int\limits_{\TT^n}|D_Q(x)f(x)|\,dx
\leq
\int\limits^1_0 D_Q^*(t)f^*(t)\,dt.
\end{equation}
In \cite[p. 98 Lemma 5]{NED} it has been shown that
\begin{equation}
\label{EQ:DQ-star-above}
		D^*_Q(t)
\lesssim
\frac{|Q|^{\frac1p}}{t^{\frac1{p'}}},
\end{equation}
where $D_Q(x)=\sum\limits_{k\in Q}e^{2\pi i k\cdot x},\,\,Q\subset \M_a$.
Hence, we get
\begin{equation}\label{EQ:aux1}
\left|
\left(D_Q,f\right)_{L^2(\TT^n)}
\right|
\leq
|Q|^{\frac1p}
\int\limits^1_0 t^{\frac1p}f^*(t)\frac{dt}{t}
=
|Q|^{\frac1p}
\|f\|_{L^{p,1}(\TT^n)}.
\end{equation}
In \cite[p. 220 Theorem 4.7]{Grafakos_2008} it has been shown that the following equality holds true
\begin{equation}
\label{EQ:Lpq-duality}
\|f\|_{L^{p,q}(X,\mu)}
=
\sup_{\substack{g\neq 0}}
\frac{
	\int\limits_{X}|fg|d\mu
	}
	{
		\|g\|_{L^{p',q'}(X,\mu)}
	}.
\end{equation}
Using the $L^{p,q}$-space duality \eqref{EQ:Lpq-duality} and \eqref{EQ:aux1}, we get
\begin{equation*}
\|D\|_{L^{p',1}(\TT^n)}
\leq |Q|^{\frac1p}.
\end{equation*}
The application of this and of the embedding propeties of the Lorenz spaces (see  \cite[p.217 Proposition 4.2]{Grafakos_2008}) yield
\begin{equation}
\|D_Q\|_{L^{p'}(\TT^n)}
\leq
\|D_Q\|_{L^{p',q}(\TT^n)}
\leq
|Q|^{\frac1p},\quad Q\in M_a.
\end{equation}

Finally, using this,  we obtain
\begin{equation}
\sup_{k\in\ZZ^n}
k^{\frac1{p'}}
\sup_{\substack{Q\in \M_a \\ |Q|\geq k}}
\frac1{|Q|}
\|
D_Q
\|_{L^{p'}(\TT^n)}
\leq
\sup_{k\in\ZZ^n}
k^{\frac1{p'}}
\sup_{\substack{Q\in \M_a \\ |Q|\geq k}}
\frac1{|Q|^{\frac1{p'}}}
=
1.
\end{equation}
This completes the proof.
\end{proof}
\begin{proof}[Proof of Theorem \ref{THM:characterization}]
We shall show that either of inequalities  \eqref{EQ:DQ-condition} and \eqref{EQ:N_pq-L_p} implies each other.
Let us first claim that we have
\begin{equation}
\label{EQ:DK_norm_implies_Npq:step_2}
(f,D_Q)_{L^2(G/K)}
=
\sum\limits_{\pi\in Q}\dpi\Tr\widehat{f}(\pi).
\end{equation}
If we assume this claim for the moment, the proof proceeds as follows.

$\Rightarrow$.
By H\"older inequality, we have
\begin{equation}
\label{EQ:obvious_step-1}
\left|
f\ast D_Q(0)
\right|
=
		\left|
		(f,D_Q)_{L^2(G/K)}
		\right|
\leq
	\|f\|_{L^p(G/K)}{\|D_Q\|_{L^{p'}(G/K)}}.
\end{equation}
We multiply the left-hand side in \eqref{EQ:obvious_step-1}  by
$
\lambda^{\frac{1}{p'}}_{\pi}
/
\nu_{\Gh_0}(Q)
$
to get
\begin{equation}
\label{EQ:obvious_step-2}
\frac
	{
		\lambda^{\frac{1}{p'}}_{\pi}
	}
	{
		\nu_{\Gh_0}(Q)
	}
		\left|
		(f\ast D_Q)(0)
		\right|
\leq
\frac
	{
		\lambda^{\frac{1}{p'}}_{\pi}
	}
	{
		\nu_{\Gh_0}(Q)
	}
\|D_Q\|_{L^{p'}(G/K)}
\|f\|_{L^p(G/K)}.
\end{equation}
Fixing $\pi\in\Gh_0$ and then
taking supremum over all $Q\in\M$ such that $\nu_{\Gh_0}(Q)\geq \lambda_{\pi}$ in the right-hand side in \eqref{EQ:obvious_step-2}, we get
\begin{equation}
\label{EQ:obvious_step-3}
\frac
	{
\lambda^{\frac{1}{p'}}_{\pi}
	}
	{
		\nu_{\Gh_0}(Q)
	}
		\left|
		(f\ast D_Q)(0)
		\right|
\leq
\left[
\lambda^{\frac{1}{p'}}_{\pi}
\sup_{\substack{Q\in \M \\ \nu_{\Gh_0}(Q) \geq \lambda_{\pi}}}
\frac
	{
		\|D_Q\|_{L^{p'}(G/K)}
	}
	{
		\nu_{\Gh_0}(Q)
	}	
\right]
\|f\|_{L^p(G/K)}.
\end{equation}
Again taking supremum over all $\pi\in\Gh_0$ in the right-hand side in \eqref{EQ:obvious_step-3}, we finally obtain
\begin{equation}
\label{EQ:obvious_step-4}
\frac
	{
		\lambda^{\frac{1}{p'}}_{\pi}
	}
	{
		\nu_{\Gh_0}(Q)
	}
		\left|
		(f\ast D_Q)(0)
		\right|
\leq
\left[
\sup_{\pi\in\Gh_0}
		\lambda^{\frac{1}{p'}}_{\pi}
\sup_{\substack{Q\in \M \\ \nu_{\Gh_0}(Q) \geq \lambda_{\pi}}}
\frac
	{
		\|D_Q\|_{L^{p'}(G/K)}
	}
	{
		\nu_{\Gh_0}(Q)
	}	
\right]
\|f\|_{L^p(G/K)}.
\end{equation}
Applying  the preceding procedure of taking the supremum on the left-hand side in \eqref{EQ:obvious_step-4}, we show
\begin{multline}
\|\widehat{f}\|_{N_{p',\infty}(\Gh_0,\M)}
=
\sup_{\pi\in\Gh_0}
\sup_{\substack{Q\in \M \\ \nu_{\Gh_0}(Q)\geq  \lambda_{\pi}}}
\frac
	{
		\lambda^{\frac{1}{p'}}_{\pi}
	}
	{
		\nu_{\Gh_0}(Q)
	}
		\left|
		(f\ast D_Q)(0)
		\right|
\|f\|_{L^p(G/K)}
\\\leq
C_{p\M}
\|f\|_{L^p(G/K)},
\end{multline}
where $C^p_{D\M}$ is the constant defined in the hypothesis of Theorem \ref{THM:characterization}.

$\Leftarrow$. 

By the definition of the $\|\cdot\|_{N_{p',\infty}(\Gh_0,\M)}$-norm, it follows from \eqref{EQ:N_pq-L_p} that
\begin{equation}
\sup_{\pi\in\Gh_0}
\sup_{\substack{Q\in \M \\ \nu_{\Gh_0}(Q)\geq    \lambda_{\pi}}}
\frac
	{
		\lambda^{\frac{1}{p'}}_{\pi}
	}
	{
		\nu_{\Gh_0}(Q)
	}
		\left|
		(f\ast D_Q)(0)
		\right|
\lesssim
\|f\|_{L^p(G/K)},\quad f\in L^p(G/K).
\end{equation}
Then  for any  pair $(\pi,Q)\in\Gh_0\times \M$ such that $\nu_{\Gh_0}(Q)\geq \lambda_{\pi}$ and $f\in L^p(G/K)$ we have
\begin{equation}
\label{EQ:obvious_step-5}
\frac
	{
		\lambda^{\frac{1}{p'}}_{\pi}
	}
	{
		\nu_{\Gh_0}(Q)
	}
		\left|
		(f,D_Q)_{L^2(G/K)}
		\right|
\leq
C
\|f\|_{L^p(G/K)},
\end{equation}
and we used the fact that $f\ast D_Q(0)=(f,D_Q)_{L^2(G/K)}$.
Multiplying both sides of \eqref{EQ:obvious_step-5}, we get
\begin{equation}
		\left|
		(f,D_Q)_{L^2(G/K)}
		\right|
\leq
C
\frac
	{
		\nu_{\Gh_0}(Q)
	}
	{
		\lambda^{\frac{1}{p'}}_{\pi}
	}
	\|f\|_{L^p(G/K)}.
\end{equation}
The inverse H\"older inequality then implies that
\begin{equation}
\|D_Q\|_{L^{p'}(G/K)}
\leq
C
\frac
	{
		\nu_{\Gh_0}(Q)
	}
	{
		\lambda^{\frac{1}{p'}}_{\pi}
	}.
\end{equation}
Equivalently, we have
\begin{equation}
\label{EQ:obvious_step-6}
\frac
	{
		\lambda^{\frac{1}{p'}}_{\pi}
	}
	{
		\nu_{\Gh_0}(Q)
	}
\|D_Q\|_{L^{p'}(G/K)}
\leq 
C.
\end{equation}
Taking supremum in the left-hand side in \eqref{EQ:obvious_step-6}, we obtain
\begin{equation}
C_{p\M}
=
\sup_{\pi\in\Gh_0}
		\lambda^{\frac{1}{p'}}_{\pi}
\sup_{\substack{Q\in \M \\ \nu_{\Gh_0}(Q)\geq \lambda_{\pi}}}
\frac
	{
		\|D_Q\|_{L^{p'}(G/K)}
	}
	{
		\nu_{\Gh_0}(Q)
	}
\leq
C<+\infty,
\end{equation}
in view of the fact that  $(\pi,Q)$ is any pair satisfying $\mu(Q)\geq \lambda_{\pi}$  and $C$ is fixed and does not depend on $Q$ nor on $\pi$.

Now, it remains to establish \eqref{EQ:DK_norm_implies_Npq:step_2}.
Since the trace $\Tr$ is invariant under taking Hermittian conjugate $\pi \mapsto \pi^*$, up to complex conjugation i.e.
$$
\Tr\pi^*
=
\overline{\Tr\pi},
$$
we have
\begin{multline}
\label{EQ:conj_compl_herm}
(g,D_Q)_{L^2(G/K)}
=
\int\limits_{G/K}g(x)\overline{D_{Q}(x)}\,dx
=
\int\limits_{G/K}g(x)\overline{\sum\limits_{\pi\in Q}\dpi\Tr\pi(x)}\,dx
\\=
\int\limits_{G/K}g(x)\sum\limits_{\pi\in Q}\dpi\overline{\Tr\pi(x)}\,dx
=
\int\limits_{G/K}g(x)\sum\limits_{\pi\in Q}\dpi\Tr\pi^*(x)\,dx
\\=
\sum\limits_{\pi\in Q}
\int\limits_{G/K}g(x)\,\dpi\Tr\pi^*(x)\,dx,
\end{multline}
where $D_Q(x)=\sum\limits_{\pi\in Q}\dpi\Tr\pi(x)$.
Interchanging $\int\limits_{G/K}$ and $\Tr$ in the last line in \eqref{EQ:conj_compl_herm}, we get
\begin{equation}
(g,D_Q)_{L^2(G/K)}
=
\sum\limits_{\pi\in Q}
\dpi\Tr \int\limits_{G/K}g(x)\pi^*(x)\,dx
=
\sum\limits_{\pi\in Q}
\dpi\Tr\widehat{g}(\pi),
\end{equation}
where we used that the Fourier coefficients $\widehat{g}(\pi)$ are, by definition, equal to
\begin{equation}
\widehat{g}(\pi)
=
\int\limits_{G/K}
g(x)\pi^*(x)\,dx.
\end{equation}
This proves \eqref{EQ:DK_norm_implies_Npq:step_2}. 
This completes the proof.
\end{proof}


\section{On some properties of net spaces}
\label{SEC:netspaces}

We now formulate some assumptions which will allow us to establish interpolation theory of $N_{p,q}$ spaces which was needed in the proof of Theorem \ref{THM:HL-Npq-version}. In the case $\Gamma=\Gh$ or $\Gamma=\Gh_0$, these assumptions will be
satisfied.

\begin{assumption}\label{assumption_1} Suppose that a positive net $\{\lambda_{\pi}\}_{\pi\in\Gamma}$ is monotone increasing, i.e.
	\begin{equation}
\xi\prec \pi \quad \text{ if and only if} \quad \lambda_{\xi}\leq \lambda_{\pi}.
	\end{equation} 
\end{assumption}
\begin{assumption} Let $\beta\in \RR$ with $\beta\neq -1$. Suppose that the following formulae are true
\label{assumption_3}
\begin{eqnarray}
\label{EQ:assumption_3_1}
	\sum\limits_{
	\substack{
	\theta\in\Gamma 
	\\ 
	\lambda_{\theta}\leq \lambda_{\pi}
			}
		    }
	\lambda^{\beta}_{\theta}\kappa_{\theta}\delta_{\theta}
	=
	C_{\beta}
	\lambda_{\pi}^{\beta+1}\text{ for } \beta> -1,
\\
\label{EQ:assumption_3_2}
\sum\limits_{
	\substack{
	\theta\in\Gamma 
	\\ 
	\lambda_{\theta}\geq \lambda_{\pi}
			}
		    }
	\lambda^{\beta}_{\theta}\kappa_{\theta}\delta_{\theta}
	=
	C_{\beta}
	\lambda_{\pi}^{\beta+1}\text{ for } \beta<-1.
\end{eqnarray}
where $C_{\beta}$ is a constant depending on $\beta$.
\end{assumption}

\begin{rem}
\label{REM:Assumption_3_G}
Assumption \ref{assumption_3} is satisfied for $\Gamma=\Gh_0, \lambda_{\theta}=\jp\theta^n,\delta_{\theta}=d_{\theta},\kappa_{\theta}=k_{\theta},\,\theta\in\Gh_0$, where $\jp\theta$ are the eigenvalues of the first-order elliptic pseudo-differential operator $(I-\Delta_{G/K})^{\frac{1}2}$ on the compact manifold $G/K$ of dimension $n$, namely, we have
\begin{eqnarray}
\label{EQ:assumption_3_1-GK}
	\sum\limits_{
	\substack{
	\theta\in\Gh_0 
	\\ 
	\jp{\theta}\leq \jp{\pi}
			}
		    }
	\jp\theta^{\beta n}  k_{\theta} d_{\theta}
	\simeq
	\jp\pi^{(\beta+1) n}\text{ for } \beta> -1,
\\
\label{EQ:assumption_3_2-compact-case}
\sum\limits_{
	\substack{
	\theta\in\Gh_0 
	\\ 
	\jp{\theta}\geq \jp{\pi}
			}
		    }
	\jp\theta^{\beta n}  k_{\theta} d_{\theta}
	\simeq
	\jp\pi^{(\beta+1) n}\text{ for } \beta<-1.
\end{eqnarray}
This will be proved in Section \ref{APP1}.
We can also recall that if $K=\{e\}$ and hence $G/K=G$ is a compact Lie group, then
$\Gh_0=\Gh$ and
$k_\theta=d_\theta$ is the dimension of the representation $[\theta]\in\Gh$.
\end{rem}		
Now we can give a proof of Corollary \ref{COR:HL-converse-SU2}.
\begin{proof}[Proof of Corollary \ref{COR:HL-converse-SU2}]
For $G=\SU2$ and $\lambda_l=(2l+1)^3$, it is straightforward to check that Assumption \ref{assumption_3} holds true. Then by Theorem \ref{THM:HL-Npq-version}, we get
\begin{equation}
\label{EQ:COR:almost}
\|\widehat{f}\|_{N_{p',p}(\frac12\NN_0,\M_1)}
\leq
C_p
\|f\|_{L^p(\SU2)},\quad 1<p\leq 2.
\end{equation}
By Definition \ref{DEF:Npq-lattice} with $\Gamma=\widehat{\SU2},\delta_l=\kappa_l=(2l+1)$, the left-hand side in \eqref{EQ:COR:almost} is equal to
\begin{equation}
\sum\limits_{\xi\in\frac12\NN_0}(2\xi+1)^{\frac{5p}2-4}
\left(
\sup_{\substack{k\in\frac12\NN_0 \\ (2k+1) \geq (2\xi+1)}}
\frac1{(2k+1)^3}
\left|
\sum\limits_{\substack{l\in\frac12\NN_0 \\ 2l+1\leq 2k+1}}
(2l+1)\Tr\widehat{f}(l)
\right|
\right)^p.
\end{equation}
Thus, we have established inequality \eqref{EQ:HL-converse-SU2}.
This completes the proof.
\end{proof}

\begin{lem}\label{REM:mean_monot_decreasing}
	Suppose that Assumption \ref{assumption_1} holds true. Then
	the averaging $\overline{F}(\lambda_{\pi})$ of $F$ is a {\it monotone decreasing} net.
\end{lem}
\begin{proof}Let $\pi,\xi,\theta\in\Gamma$ and $\pi\succ \xi$. We will show that $\overline{F}(\lambda_{\pi})\leq \overline{F}(\lambda_{\xi})$.
Since $\lambda=\{\lambda_{\pi}\}_{\pi\in\Gamma}$ is a monotone increasing net, we have
$$
	\{
	Q\in\M 
	\colon
	\nu_\Gamma(Q)\geq \lambda_{\pi}
	\}
	\subset
	\{
	Q\in\M 
	\colon
	\nu_\Gamma(Q)\geq \lambda_{\xi}		
	\}.
$$
Therefore, we get
$$
\overline{F}(\lambda_{\pi})
=
\sup\limits_{\substack{
		Q\in \M \\ \nu_\Gamma(Q)\geq \lambda_{\pi}
		}}
		\frac1{\nu_\Gamma(Q)}
		\left|\sum\limits_{\theta\in Q}d_{\theta}\Tr F(\theta)\right|
\leq
\sup\limits_{\substack{
		Q\in \M \\ \nu_\Gamma(Q)\geq \lambda_{\xi}
	}}
	\frac1{\nu_\Gamma(Q)}
	\left|\sum\limits_{\theta\in Q}d_{\theta}\Tr F(\theta)\right|	
	=
\overline{F}(\lambda_{\xi}).	
$$
This completes the proof.
\end{proof}

\begin{thm}
	\label{THM:Net_Space_Embedd}
\begin{enumerate}
\item 
Let $1\leq p<\infty$, $1<q\leq\infty$ and let $\M_1\subset\M_2$ be two arbitrary fixed collections of finite subsets of $\Gamma$, then we have
\begin{equation}
\label{EQ:Net_space_Embedd_M}
N_{p,q}(\M_2)
\hookrightarrow
N_{p,q}(\M_1).
\end{equation}
\item 
Let $1\leq p <\infty$ and $1\leq q_1 \leq q_2 \leq \infty$. Suppose that Assumptions \ref{assumption_1} and \ref{assumption_3} hold true and  $\M$ is a fixed arbitrary collection of finite subsets of  $\Gamma$. Then we have the following embedding
\begin{equation}
\label{EQ:Net_Space_Embedd}
N_{p,q_1}(\Gamma,\M) \hookrightarrow N_{p,q_2}(\Gamma,\M).
\end{equation}
\end{enumerate}
\end{thm}
\begin{proof}[Proof of Theorem \ref{THM:Net_Space_Embedd}]
First, we notice that \eqref{EQ:Net_space_Embedd_M} follows directly from the definition.
Therefore, we concentrate on proving the second part of the Theorem.

Let first $1\leq q_1 \leq q_2 <\infty$.
	By definition, we have
	\begin{equation}
	\label{EQ:Proof_Net_Space_Embed_aux_1}
		\|F\|_{N_{p,q_2}(\Gamma,\M)}
		=
		\left(
		\sum_{\pi\in\Gamma}
		\left(
		\lambda^{\frac{1}{p}}_{\pi}
		\overline{F}[\lambda_{\pi}]
		\right)^{q_2}\frac{\delta_{\pi}\kappa_{\pi}}{\lambda_{\pi}}
		\right)^{\frac1{q_2}}.		
	\end{equation}
 Using formula \eqref{EQ:assumption_3_1} from Assumption \ref{assumption_3} with $\beta=q_1(\frac1p+\eps)-1$, we get,
 with a sufficiently large $\eps>0$,
	$$
	\lambda_{\pi}^{\frac{q_2}{p}}
	=
	\lambda^{-\eps q_2}_{\pi}\lambda^{q_2(\frac1p+\eps)}_{\pi}
	=
	C^{\frac{q_2}{q_1}}_{\beta}
	\lambda^{-\eps q_2}_{\pi}
	\left(
	\sum\limits_
	{
		\substack
		{
			\theta\in\Gamma
			\\
			\lambda_{\theta}
			\leq 
			\lambda_{\pi}
		}
	}
	\lambda^{q_1(\frac{1}{p}+\eps)-1}_{\theta}\delta_{\theta}\kappa_{\theta}
	\right)^{\frac{q_2}{q_1}}.
	$$
Thus, up to a constant, the expression in \eqref{EQ:Proof_Net_Space_Embed_aux_1} equals to
\begin{equation}\label{EQ:aux-ns1}
	\left(
	\sum_{\pi\in\Gamma}
	\lambda_{\pi}^{-\varepsilon q_2}
	\left(\overline{F}[\lambda_{\pi}]\right)^{q_2}
	\left(
	\sum
	\limits_{
		\substack
		{
			\theta\in\Gamma
			\\
			\lambda_{\theta}\leq \lambda_{\pi}
		}
		}
			\lambda_{\theta}^{q_1(\frac1p+\eps)}
			\frac{\delta_{\theta}\kappa_{\theta}}{\lambda_{\theta}}
	\right)^{\frac{q_2}{q_1}}
	\frac{\delta_{\pi}\kappa_{\pi}}{\lambda_{\pi}}
	\right)^{\frac1{q_2}}.
\end{equation}
		
	In view of Lemma \ref{REM:mean_monot_decreasing} the averaging function 
	$\overline{F}(\lambda_{\pi})$ is a monotone decreasing net. Therefore, \eqref{EQ:aux-ns1} does not exceed
$$
	\left[
		\left(
		\sum_{\pi\in\Gamma}
		\lambda_{\pi}^{-\varepsilon q_2}
		\left(
		\sum
		\limits_{
			\substack
			{
				\theta\in\Gamma
				\\
				\lambda_{\theta}\leq \lambda_{\pi}
			}
		}
		\lambda_{\theta}^{q_1(\frac1p+\eps)}
		\overline{F}[\lambda_{\theta}]^{q_1}
		\frac{\delta_{\theta}\kappa_{\theta}}{\lambda_{\theta}}
		\right)^{\alpha}
		\frac{\delta_{\pi}\kappa_{\pi}}{\lambda_{\pi}}
		\right)^{\frac{1}{\alpha}}
		\right]^{\frac1{q_1}},
$$
	where $\alpha=\frac{q_2}{q_1}$.
We have thus proved that
\begin{equation}
\label{EQ:Proof_Net_Space_Embedd_aux_2}
\|F\|_{N_{p,q_2}(\Gamma,\M)}
\leq
C^{\frac1{q_1}}_{\beta}
	\left[
	\left(
	\sum_{\pi\in\Gamma}
	\lambda_{\pi}^{-\varepsilon q_2}
	\left(
	\sum
	\limits_{
		\substack
		{
			\theta\in\Gamma
			\\
			\lambda_{\theta}\leq \lambda_{\pi}
		}
	}
	\lambda_{\theta}^{q_1(\frac1p+\eps)}
	\overline{F}[\lambda_{\theta}]^{q_1}
	\frac{\delta_{\theta}\kappa_{\theta}}{\lambda_{\theta}}
	\right)^{\alpha}
	\frac{\delta_{\pi}\kappa_{\pi}}{\lambda_{\pi}}
	\right)^{\frac{1}{\alpha}}
	\right]^{\frac1{q_1}}.	
\end{equation}
	Now, we consider an $\ell^{\alpha}(\Gamma,\omega)$ space with the measure $\omega$ defined as follows
	$$
	\omega_{\Gamma}(\pi)=\lambda_{\pi}^{-\eps q_2-1}\delta_{\pi}\kappa_{\pi}.
	$$
	Denote  by $I$ the right-hand side in \eqref{EQ:Proof_Net_Space_Embedd_aux_2}.  Then
	using the $\ell^{\alpha}(\Gamma,\omega_{\Gamma})$-space duality, 
	we have
	$$
	I^{q_1}=
	\sup_{\substack{b\in\ell^{\alpha'}(\Gamma,\omega_{\Gamma})\\ \|b\|_{\ell^{\alpha'}}=1}}
	\left|
		\sum_{\pi\in\Gamma}
		\left(
		\sum
		\limits_{
			\substack
			{
				\theta\in\Gamma
				\\
				\lambda_{\theta}\leq \lambda_{\theta}
			}
		}
		\overline{F}[\lambda_{\theta}]^{q_1}
\lambda^{q_1(\frac1p+\eps)-1}_{\theta}\delta_{\theta}\kappa_{\theta}
		\right)
		\overline{b_{\pi}}
\omega_{\Gamma}(\pi)
\right|,
	$$
where $\alpha'$ denotes the exponent conjugate to $\alpha$ so that  $\frac1{\alpha}+\frac1{\alpha'}=1$.
Then using Fubini theorem and H\"older inequality, we have
\begin{multline*}
\left|
		\sum_{\pi\in\Gamma}
		\left(
		\sum
		\limits_{
			\substack
			{
				\theta\in\Gamma
				\\
				\lambda_{\theta}\leq \lambda_{\pi}
			}
		}
		\overline{F}[\lambda_{\theta}]^{q_1}
\lambda^{q_1(\frac1p+\eps)-1}_{\theta}\delta_{\theta}\kappa_{\theta}
		\right)
		\overline{b_{\pi}}
		\omega_{\Gamma}(\pi)				
	\right|
\\ \leq
	\sum_{\pi\in\Gamma}
	\left(
	\sum
	\limits_{
		\substack
		{
			\theta\in\Gamma
			\\
			\lambda_{\theta}\leq \lambda_{\pi}
		}
	}
\overline{F}[\lambda_{\theta}]^{q_1}
\lambda^{q_1(\frac1p+\eps)-1}_{\theta}\delta_{\theta}\kappa_{\theta}
	\right)
	\left|
	b_{\pi}
	\right|
	\omega_{\Gamma}(\pi)
\\ =
	\sum_{\theta\in\Gamma}
	\overline{F}[\lambda_{\theta}]^{q_1}
	\lambda^{q_1(\frac1p+\eps)-1}_{\theta}\delta_{\theta}\kappa_{\theta}	
	\sum
	\limits_{
		\substack
		{
			\pi\in\Gamma
			\\
			\lambda_{\pi}\geq \lambda_{\theta}
		}
	}	
	\left|
	b_{\pi}
	\right|
	\omega_{\Gamma}(\pi)
\\
\leq
	\sum_{\theta\in\Gamma}
\overline{F}[\lambda_{\theta}]^{q_1}
\lambda^{q_1(\frac1p+\eps)-1}_{\theta}\delta_{\theta}\kappa_{\theta}
	\left(
	\sum
	\limits_{
		\substack
		{
			\pi\in\Gamma
			\\
			\lambda_{\pi}\geq \lambda_{\theta}
		}
	}	
	\omega_{\Gamma}(\pi)
	\right)^{\frac1{\alpha}}
\left(
	\sum
	\limits_{
		\substack
		{
			\pi\in\Gamma
			\\
			\lambda_{\pi}\geq \lambda_{\theta}
		}
	}		
\left|
b_{\pi}
\right|^{\alpha'}
\omega_{\Gamma}(\pi)
\right)^{\frac1{\alpha'}}
\\\leq
\left[
	\sum_{\theta\in\Gamma}
	\overline{F}[\lambda_{\theta}]^{q_1}
	\left(
	\sum
	\limits_{
		\substack
		{
			\pi\in\Gamma
			\\
			\lambda_{\pi}\geq \lambda_{\theta}
		}
	}	
	\omega_{\Gamma}(\pi)
	\right)^{\frac1{\alpha}}
\lambda^{q_1(\frac1p+\eps)-1}_{\theta}\delta_{\theta}\kappa_{\theta}
	\right]\cdot
	\|b\|_{\ell^{\alpha'}(\Gamma,\omega_{\Gamma})}.					
\end{multline*}
Thus, using this and taking supremum over all $b\in \ell^{\alpha'}(\Gamma,\omega_{\Gamma})$, we get from \eqref{EQ:Proof_Net_Space_Embedd_aux_2} 
$$
\|F\|_{N_{p,q_2(\Gamma,\M)}}
\leq
C^{\frac1{q_1}}_{\beta}
\left(
	I^{q_1}
	\right)^{\frac1{q_1}}
\leq
C^{\frac1{q_1}}_{\beta}
	\left[
	\sum_{\theta\in\Gamma}
	\overline{F}[\lambda_{\theta}]^{q_1}
	\left(
	\sum
	\limits_{
		\substack
		{
			\pi\in\Gamma
			\\
			\lambda_{\pi}\geq \lambda_{\theta}
		}
	}	
	\omega_{\Gamma}(\pi)
	\right)^{\frac1{\alpha}}
\lambda^{q_1(\frac1p+\eps)-1}_{\theta}\delta_{\theta}\kappa_{\theta}
	\right]^{\frac1{q_1}}.
$$
Again, using formula \eqref{EQ:assumption_3_2} with $\beta=-\eps q_2-1$ from Assumption \ref{assumption_3} and recalling 
that $\alpha=\frac{q_2}{q_1}$, we get
$$
\left(
		\sum
		\limits_{
			\substack
			{
				\pi\in\Gamma
				\\
				\lambda_{\pi}\geq \lambda_{\theta}
			}
		}	
		\omega_{\Gamma}(\pi)
		\right)^{\frac1{\alpha}}
=
\left(
\sum
\limits_{
	\substack
	{
		\pi\in\Gamma
		\\
		\lambda_{\pi}\geq \lambda_{\theta}
	}
}	
\lambda^{-\eps q_2-1}_{\pi}\delta_{\pi}\kappa_{\pi}
\right)^{\frac{q_1}{q_2}}	
=
C^{\frac{q_1}{q_2}}_{\beta}
\lambda^{-\eps q_1}_{\theta},	
$$
for sufficiently large $\eps$.
Finally, we have
\begin{multline*}
\|F\|_{N_{p,q_2}(\Gamma,\M)}
\leq
C^{\frac{2}{q_1}}_{\beta}
\left[
\sum_{\theta\in\Gamma}
\overline{F}[\lambda_{\theta}]^{q_1}
\lambda^{-\eps q_1}_{\theta}
\lambda^{q_1(\frac1p+\eps)-1}\delta_{\theta}\kappa_{\theta}
\right]^{\frac1{q_1}}
\\=
C^{\frac{2}{q_1}}_{\beta}
\left[
\sum_{\theta\in\Gamma}
\left(
\lambda^{\frac1p}_{\pi}
\overline{F}[\lambda_{\theta}]
\right)
^{q_1}
\frac{\delta_{\theta}\kappa_{\theta}}{\lambda_{\theta}}
\right]^{\frac1{q_1}}
=
C^{\frac{2}{q_1}}_{\beta}
\|F\|_{N_{p,q_1}(\Gamma,\M)}.
\end{multline*}
For $q_2=\infty$ and $q_1<q_2=\infty$, using Assumption \ref{assumption_3}, we have
\begin{multline*}
\|F\|_{N_{p,\infty}(\Gamma,\M)}
=
\sup_{\pi\in\Gamma}
\lambda^{\frac1p}_{\pi}
\overline{F}[\lambda_{\pi}]
=
C^{\frac1{q_1}}_{\frac{q_1}p-1}
\sup_{\pi\in\Gamma}
\left(
\sum\limits_{\substack{\theta\in\Gamma \\ \lambda_{\theta}\leq\lambda_{\pi}}}
\lambda^{\frac{q_1}{p}}_{\theta}
\overline{F}[\lambda_{\pi}]^{q_1}
\frac{\kappa_{\theta}\delta_{\theta}}{\lambda_{\theta}}
\right)^{\frac1{q_1}}
\\\lesssim
\sup_{\pi\in\Gamma}
\left(
\sum\limits_{\substack{\theta\in\Gamma \\ \lambda_{\theta}\leq\lambda_{\pi}}}
\lambda^{\frac{q_1}{p}}_{\theta}
\overline{F}[\lambda_{\theta}]^{q_1}
\frac{\kappa_{\theta}\delta_{\theta}}{\lambda_{\theta}}
\right)^{\frac1{q_1}}
\leq
\left(
\sum\limits_{\substack{\theta\in\Gamma }}
\lambda^{\frac{q_1}{p}}_{\theta}
\overline{F}[\lambda_{\theta}]^{q_1}
\frac{\kappa_{\theta}\delta_{\theta}}{\lambda_{\theta}}
\right)^{\frac1{q_1}}=\|F\|_{N_{p,q_1(\Gamma,\M)}},
\end{multline*}
where in the first inequality we used the monotonicity of the averaging function $\overline{F}$ established in Lemma \ref{REM:mean_monot_decreasing}.
This completes the proof.
\end{proof}

We now establish interpolation properties of net spaces.
Let $(A_0,A_1)$ be a compatible pair of Banach spaces (cf. \cite{BL2011}) and let
$$
K(t,F;A_0,A_1):=
\inf_{F=F_0+F_1}
\left(
\|F_0\|_{A_0}+t\|F_1\|_{A_1}
\right),\quad  F\in A_0+A_1,
$$
be the Peetre functional.
\begin{thm}
	\label{net_space_interpolation}
	Let $1\leq p_1 < p_2 <\infty,\,1\leq q_1,q_2,q \leq \infty,\,0<\theta<1$.
	Suppose that Assumptions \ref{assumption_1} and \ref{assumption_3} hold true.
	 Then we have
	\begin{equation}
	\label{EQ:Npq-interpolation}
	(N_{p_1,q_1}(\Gamma,\M),N_{p_2,q_2}(\Gamma,\M))_{\theta,q}\hookrightarrow N_{p,q}(\Gamma,\M),
	\end{equation}
	where $\frac1p=\frac{1-\theta}{p_1}+\frac{\theta}{p_2}$.
\end{thm}

\begin{proof}[Proof of Theorem \ref{net_space_interpolation}]

	Since by Theorem \ref{THM:Net_Space_Embedd} we have
	$N_{p_i,q_i}\hookrightarrow N_{p_i,\infty}$, $i=1,2$, it is sufficient to prove
	$$
	(N_{p_1,\infty},N_{p_2,\infty})_{\theta,q}\hookrightarrow N_{p,q}.
	$$
	Assume first that $q<+\infty$.
	Let $\pi\in\Gamma, F=F_1+F_2, F_1\in N_{p_1,\infty}$ and $F_2\in N_{p_2,\infty}$. It is clear that
	$$
	\overline{F}[\lambda_{\pi}] \leq \overline{F_1}[\lambda_{\pi}] +\overline{F_2}[\lambda_{\pi}].
	$$
	Denote $v(t):=t^{\frac1{\frac1{p_1}-\frac1{p_2}}}$, $t>0$. It is obvious that
	\begin{multline*}
	\sup\limits_{\substack{\xi\in\Gamma\\ \lambda_{\xi} \leq v(t)}}\lambda_{\xi}^{\frac{1}{p_1}}\overline{F}[\lambda_{\xi}]
	\leq
	\sup\limits_{\substack{\xi\in\Gamma}}\lambda_{\xi}^{\frac{1}{p_1}}\overline{F_1}[\lambda_{\xi}]
	+
	\sup\limits_{\substack{\xi\in\Gamma\\ \lambda_{\xi} \leq v(t)}}\lambda_{\xi}^{\frac{1}{p_1}-\frac1{p_2}+\frac1{p_2}}\overline{F_2}[\lambda_{\xi}]
	\\\leq
	\sup\limits_{\substack{\xi\in\Gamma}}\lambda_{\xi}^{\frac{1}{p_1}}\overline{F_1}[\lambda_{\xi}]
	+
	t
	\sup\limits_{\substack{\xi\in\Gamma}}\lambda_{\xi}^{\frac{1}{p_2}}\overline{F_2}[\lambda_{\xi}].
	\end{multline*}
	Taking the infimum over all possible reprsentations $F=F_1+F_2$, we obtain
	\begin{equation}
	\label{aux_petre_estimate}
	\sup\limits_{\substack{\xi\in\Gamma\\ \lambda_{\xi} \leq v(t)}}\lambda_{\xi}^{\frac{1}{p_1}}\overline{F}[\lambda_{\xi}]
	\lesssim
	K(t,F;N_{p_1,\infty},N_{p_2,\infty}).
	\end{equation}

Thus, making a substitution $t \rightarrow t^{\frac1{p_1}-\frac1{p_2}}$, we have
\begin{multline}
\label{EQ:interm-1}
	\int\limits^{+\infty}_{0}\left(t^{-\theta}K(t,F)\right)^q\frac{dt}{t}
\geq
\int\limits^{+\infty}_0
\left(
t^{-\theta}\sup_{\substack{\pi\in\Gamma \\ \lambda_{\pi}\leq v(t)}}\lambda_{\pi}^{\frac1{p_1}}\overline{F}[\lambda_{\pi}]
\right)^{q}\frac{dt}{t}
\\\cong
\int\limits^{+\infty}_{0}\left(t^{-\theta(\frac1{p_1}-\frac1{p_2})}\sup_{\substack{\pi\in\Gamma \\ \lambda_{\pi}\leq t}}\lambda_{\pi}^{\frac1{p_1}}\overline{F}[\lambda_{\pi}]\right)^q\frac{dt}{t}.
\end{multline}
Decomposing
$$
	(0,+\infty)=\bigsqcup_{s\in\ZZ}[2^s,2^{s+1}),
$$
we have
\begin{multline}
\label{EQ:net_space_interpolation_aux_1}
\int\limits^{+\infty}_{0}
\left(
t^{-\theta(\frac1{p_1}-\frac1{p_2})}
\sup_{\substack{\pi\in\Gamma \\ \lambda_{\pi}\leq t}}
\lambda_{\pi}^{\frac1{p_1}}\overline{F}[\lambda_{\pi}]
\right)^q
\frac{dt}{t}
=
\sum_{s\in\ZZ}
\int\limits^{2^{s+1}}_{2^s}
\left(
t^{-\theta(\frac1{p_1}-\frac1{p_2})}
\sup_{\substack{\pi\in\Gamma \\ \lambda_{\pi}\leq t}}
\lambda_{\pi}^{\frac1{p_1}}\overline{F}[\lambda_{\pi}]
\right)^q
\frac{dt}{t}
\\\cong
\sum_{s\in\ZZ}
\left(
2^{-s\theta(\frac1{p_1}-\frac1{p_2})}
\sup_{\substack{\pi\in\Gamma \\ \lambda_{\pi}\leq 2^s}}
\lambda_{\pi}^{\frac1{p_1}}\overline{F}[\lambda_{\pi}]
\right)^q
\int\limits^{2^{s+1}}_{2^s}\frac{dt}{t}
\cong
\sum_{s\in\ZZ}
\left(
2^{-s\theta(\frac1{p_1}-\frac1{p_2})}
\sup_{\substack{\pi\in\Gamma \\ \lambda_{\pi}\leq 2^s}}
\lambda_{\pi}^{\frac1{p_1}}\overline{F}[\lambda_{\pi}]
\right)^q
\\\geq
\sum_{s\in\ZZ}
\left(
2^{-s\theta(\frac1{p_1}-\frac1{p_2})}
\sup_{\substack{\pi\in\Gamma \\ 2^{s-1}\leq\lambda_{\pi}\leq 2^s}}
\lambda_{\pi}^{\frac1{p_1}}\overline{F}[\lambda_{\pi}]
\right)^q
\geq
\sum_{s\in\ZZ}
\left(
2^{-s\theta(\frac1{p_1}-\frac1{p_2})}
2^{(s-1)\frac1{p_1}}\overline{F}[2^{s-1}]
\right)^q.
\end{multline}
Using formulae \eqref{EQ:assumption_3_1} with $\beta=0$ from Assumption \ref{assumption_3}, we get
\begin{equation}
\label{EQ:assumption-alpha-0}
	\sum\limits_{\substack{\theta\in\Gamma \\ 2^s\leq \lambda_{\theta}\leq 2^{s+1}}}
		\frac{\delta_{\theta}\kappa_{\theta}}{\lambda_{\theta}}
		\cong
		\frac1{2^s}
		\left(
	\sum\limits_{
		\substack
		{\theta\in\Gamma \\ \lambda_{\theta}\leq 2^{s+1}}
	}
	\delta_{\theta}\kappa_{\theta}
	-
	\sum\limits_{
		\substack
		{\theta\in\Gamma \\ \lambda_{\theta}\leq 2^{s}}
	}
	\delta_{\theta}\kappa_{\theta}
	\right)
		\cong
		\frac{
		2^{s+1}-2^s
	}{2^s}
		=1.
\end{equation}

Therefore recalling that $\frac1p=\frac{1-\theta}{p_1}+\frac{\theta}{p_2}$ and combining formulae \eqref{EQ:interm-1}, \eqref{EQ:net_space_interpolation_aux_1}, \eqref{EQ:assumption-alpha-0}, we have 
\begin{multline}
	\int\limits^{+\infty}_{0}\left(t^{-\theta}K(t,F)\right)^q\frac{dt}{t}
\geq
\sum_{s\in\ZZ}
\left(
2^{\frac{s}p}
\overline{F}[2^s]
\right)^q
	\sum\limits_{\substack{\theta\in\Gamma \\ 2^s\leq \lambda_{\theta}\leq 2^{s+1}}}
		\frac{\delta_{\theta}\kappa_{\theta}}{\lambda_{\theta}}
\\
\cong
\sum_{s\in\ZZ}
\left(
2^{\frac{s+1}p}
\overline{F}[2^s]
\right)^q
	\sum\limits_{\substack{\theta\in\Gamma \\ 2^s\leq \lambda_{\theta}\leq 2^{s+1}}}
		\frac{\delta_{\theta}\kappa_{\theta}}{\lambda_{\theta}}
\geq
\sum_{s\in\ZZ}
	\sum\limits_{\substack{\theta\in\Gamma \\ 2^s\leq \lambda_{\theta}\leq 2^{s+1}}}
\left(
\lambda^{\frac1p}_{\theta}
\overline{F}[\lambda_{\theta}]
\right)^q
\frac{\delta_{\theta}\kappa_{\theta}}{\lambda_{\theta}}
\\=
\sum\limits_{\substack{\theta\in\Gamma}}
\left(
\lambda^{\frac1p}_{\theta}
\overline{F}[\lambda_{\theta}]
\right)^q
\frac{\delta_{\theta}\kappa_{\theta}}{\lambda_{\theta}}
=\|F\|_{N_{p,q}(\Gamma,\M)}.
\end{multline}
where in the last inequality we used that $\overline{F}(t)$ is a decreasing function of $t$. This proves \eqref{EQ:Npq-interpolation} for $q<+\infty$.
For $	q=\infty$, making again substitution $t\mapsto t^{\frac1{p_1}-\frac1{p_2}}$, we have
\begin{equation}
\label{EQ:q-inf}
\sup_{t>0}t^{-\theta}K(t,F)
=
\sup_{t>0}t^{-\theta(\frac1{p_1}-\frac1{p_2})}K(t^{\frac1{p_1}-\frac1{p_2}},F).
\end{equation}
Then using formula \eqref{aux_petre_estimate}, we continue inequality \eqref{EQ:q-inf} as
\begin{multline}
\geq
\sup_{t>0}t^{-\theta(\frac1{p_1}-\frac1{p_2})}\sup_{\substack{\pi\in\Gamma \\ \lambda_{\pi}\leq t}}\lambda^{\frac1{p_1}}_{\pi}\overline{F}[\lambda_{\pi}]
\gtrsim
\lambda^{-\theta(\frac1{p_1}-\frac1{p_2})}_{\xi}
\sup_{\substack{\pi\in\Gamma \\ \lambda_{\pi}\leq \lambda_{\xi}}}\lambda^{\frac1{p_1}}_{\pi}\overline{F}[\lambda_{\pi}]
\\\geq
\lambda^{-\theta(\frac1{p_1}-\frac1{p_2})}_{\xi}
\lambda^{\frac1{p_1}}_{\xi}\overline{F}[\lambda_{\xi}]
=
\lambda^{\frac1p}_{\xi}\overline{F}[\lambda_{\xi}]
,
\end{multline}
here $\lambda_{\xi}$ is an arbitrary element of the net $\lambda=\{\lambda_{\xi}\}_{\xi\in\Gamma}$.
Taking thus supremum over all $\xi\in\Gamma$, we finally obtain
\begin{equation}
\sup_{t>0}
t^{-\theta}K(t,F)
\gtrsim
\sup_{\xi\in\Gamma}
\lambda^{\frac1p}_{\xi}\overline{F}[\lambda_{\xi}]
=\|F\|_{N_{p,\infty}(\Gamma,\M)}.
\end{equation}
We have established embedding \eqref{EQ:Npq-interpolation} also for $q=\infty$. This completes the proof.
\end{proof}

\section{Proof of Remark \ref{REM:Assumption_3_G}}
\label{APP1}

Here we prove Remark \ref{REM:Assumption_3_G}.
\smallskip

We denote by $s_k$ and $m_k$ the enumerated eigenvalues of $(1-\Delta_{G/K})^{\frac12}$ and $(1-\Delta_{G/K})^{\frac{n}2}$ respectively.  We assume that they are ordered, i.e.
\begin{eqnarray}
s_1\leq s_2 \leq\ldots\leq s_{k}\leq s_{k+1} \leq \ldots,\nonumber \\
\label{PROOF:Assumption_3_G:1}
m_1\leq m_2\leq\ldots\leq m_{k}\leq m_{k+1} \leq \ldots,
\end{eqnarray}
multiplicities taken into account.
Let us also denote by $N_k(L)$ the eigenvalue counting function for the eigenvalues $\lambda^k=\{\lambda^k_i\}_{i\in\NN}$ of the $k$-th order operator $(I-\Delta_{G/K})^{\frac{k}2}$, i.e.
$$
N_k(L)=\sum\limits_{\substack{i\in\NN\\ \lambda^k_i\leq L}}  1.
$$
First, we show that
\begin{equation}
\label{PROOF:Assumption_3_G:eigenvalue_asymp}
		m_k \cong k.
\end{equation}
We use the Weyl counting function asymptotics for the first order elliptic pseudo-differential operator $(1-\Delta_{G/K})^{\frac12}$ on the compact homogeneous manifold $G/K$, to get that the eigenvalue counting function $N_1(L)$ of eigenvalues of $(1-\Delta_{G/K})^{\frac12}$ counted with multiplicities (see e.g. \cite{Shubin:BK-1987}), as
$$
	N_1(L) = \sum_{i\colon s_i \leq L} 1 \cong L^{n}.
$$
Then we get the following asymptotic for the $n$-th order elliptic pseudo-differential operator $A=(I-\Delta_{G/K})^{\frac{n}2}$
\begin{equation}
\label{PROOF:Assumption_3_G:2}
N_n(L)=
 \sum_{i\colon m_i \leq L} 1
 \cong L. 
\end{equation}
Now, we fix an arbitrary eigenvalue $m_k$ and set $L=m_k$ in \eqref{PROOF:Assumption_3_G:2} to get
\begin{equation}
\label{PROOF:Assumption_3_G:3}
	N_n(m_k) \cong m_k.
\end{equation}
Since \eqref{PROOF:Assumption_3_G:1}, we have
\begin{equation}
\label{PROOF:Assumption_3_G:4}
	N_n(m_k) = k.
\end{equation}
Combining \eqref{PROOF:Assumption_3_G:3} and \eqref{PROOF:Assumption_3_G:4}, we get  \eqref{PROOF:Assumption_3_G:eigenvalue_asymp}.
Further, we fix an arbitrary $\pi\in\Gh_0$ coresponding to a $\mu_{k_0}$ such that $\jp\pi^n=\mu_{k_0}$. Then using \eqref{PROOF:Assumption_3_G:eigenvalue_asymp}, we get
\begin{equation*}
	\sum\limits_{\substack{\xi\in\Gh_0\\\jp\xi\leq\jp\pi}}\jp\xi^{n\beta}
=
	\sum\limits_{k\colon \mu_k \leq \mu_{k_0}}\mu^{\beta}_k
\cong
	\sum\limits^{k_0}_{k=1}k^{\beta}
\cong
k^{\beta+1}_0
\cong
\mu^{\beta+1}_{k_0}
=
\jp\pi^{n(\beta+1)}.
\end{equation*}

This proves  \eqref{EQ:assumption_3_2-compact-case} with $\Gamma=\Gh_0, \lambda_{\theta}=\jp\theta^n,\delta_{\theta}=d_{\theta},\kappa_{\theta}=k_{\theta},\,\theta\in\Gh_0$.


\end{document}